\documentclass{article} 
\usepackage[utf8]{inputenc}
\usepackage[T1]{fontenc}
\usepackage{lmodern}
\usepackage[a4paper]{geometry}
\usepackage{babel}
\usepackage{enumitem}
\usepackage{fullpage}
\usepackage{pifont}

\usepackage[usenames, dvipsnames]{xcolor}
\usepackage{tikz}
\usetikzlibrary{patterns}
\usepackage[justification=justified]{caption}
\usepackage{pgfplots}
\pgfplotsset{compat=1.15}
\usepackage{mathrsfs}
\usetikzlibrary{arrows}
\usepackage{amssymb,amsmath,mathtools,amsthm,dsfont}
\usepackage{stmaryrd}

\usepackage[all]{xy}

\newtheorem{theoreme}{Théorème}[section]
\newtheorem{theorem}[theoreme]{Theorem}
\newtheorem{prop-f}[theoreme]{Proposition}

\newtheorem{corollary}[theoreme]{Corollary}

\newtheorem{remk}[theoreme]{Remark}
\newtheorem{claim}[theoreme]{Claim}

\newcommand{\E}{\mathbb{E}}

\renewcommand{\P}{\mathbb{P}}

\newcommand{\R}{\mathbb{R}}

\newcommand{\Z}{\mathbb{Z}}
\renewcommand{\l}{\ell}

\renewcommand{\L}{\mathcal{L}}

\newcommand{\cA}{\mathcal{A}}

\newcommand{\cC}{\mathcal{C}}

\newcommand{\cE}{\mathcal{E}}
\newcommand{\cF}{\mathcal{F}}

\newcommand{\cP}{\mathcal{P}}

\newcommand{\fpm}{\vec{\pi}}
\newcommand{\fcP}{\vec{\cP}}

\newcommand{\fg}{\vec{\gamma}}
\newcommand{\fm}{\vec{\mu}}
\newcommand{\ft}{\vec{t}}

\newcommand{\ulm}{u^\Lambda}
\newcommand{\vlm}{v^\Lambda}
\renewcommand{\r}{t_{\min}}

\newcommand{\Lu}{\underline{L}}

\newcommand{\PP}{\pi^+}
\newcommand{\PPP}{\pi^{++}}

\renewcommand{\epsilon}{\varepsilon}
\renewcommand{\phi}{\varphi}

\definecolor{qqwuqq}{rgb}{0,0.39215686274509803,0}
\definecolor{ccqqqq}{rgb}{0.8,0,0}

\numberwithin{equation}{section}
\newcounter{numeroexo}

\makeindex
\begin{document}
	
	\title{Strict inequality between the time constants of first-passage percolation and directed first-passage percolation}
	\author{Antonin Jacquet\footnote{Institut Denis Poisson, UMR-CNRS 7013, Université de Tours, antonin.jacquet@univ-tours.fr}}
	\date{}
	\maketitle
	\begin{abstract}
		In the models of first-passage percolation and directed first-passage percolation on $\Z^d$, we consider a family of i.i.d\ random variables indexed by the set of edges of the graph, called passage times. 
		For every vertex $x \in \Z^d$ with nonnegative coordinates, we denote by $t(0,x)$ the shortest passage time to go from $0$ to $x$ and by $\ft(0,x)$ the shortest passage time to go from $0$ to $x$ following a directed path. 
		Under some assumptions, it is known that for every $x \in \R^d$ with nonnegative coordinates, $t(0,\lfloor nx \rfloor)/n$ converges to a constant $\mu(x)$ and that $\ft(0,\lfloor nx \rfloor)/n$ converges to a constant $\fm(x)$. With these definitions, we immediately get that $\mu(x) \le \fm(x)$. 
		In this paper, we get the strict inequality $\mu(x) < \fm(x)$ as a consequence of a new exponential bound for the comparison of $t(0,x)$ and $\ft(0,x)$ when $\|x\|$ goes to $\infty$. This exponential bound is itself based on a lower bound on the number of edges of geodesics in first-passage percolation (where geodesics are paths with minimal passage time).
	\end{abstract}
	
	

	\section{Introduction and results}
	
	\subsection{The model of first-passage percolation.}\label{Sous-section modèle classique de ppp.}
	
	Let $d \ge 2$. We consider first-passage percolation on the usual undirected graph $\Z^d$. The edges are those connecting two vertices $x$ and $y$ such that $\|x-y\|_1=1$. We denote by $\cE$ the set of edges. 
	We consider a family $T=\{T(e) \, : \, e \in \cE\}$ of i.i.d.\ random variables taking values in $[0,\infty]$ and defined on a probability space $(\Omega,\cF,\P)$. The random variable $T(e)$ represents the passage time of the edge $e$. Their common distribution is denoted by $\L$, and the minimum of its support is denoted by $\r$.
	
	A finite path $\pi=(x_0,\dots,x_k)$ is a sequence of adjacent vertices of $\Z^d$, i.e.\ for all $i=0,\dots,k-1$, $\|x_{i+1}-x_i\|_1=1$. We say that $\pi$ is a path between $x_0$ and $x_k$. A path $\pi$ is a self-avoiding path if the vertices $x_0,\dots,x_k$ are all different.
	Sometimes we identify a path with the sequence of edges it visits, writing $\pi=(e_1,...,e_k)$ where for $i=1,\dots,k$, $e_i=\{x_{i-1},x_i\}$. 
	For two vertices $x$ and $y$, we denote by $\cP(x,y)$ the set of finite self-avoiding paths between $x$ and $y$.
	The passage time $T(\pi)$ of a path $\pi=(e_1,\dots,e_k)$ is defined as the sum of the variables $T(e_i)$ for $i=1,\dots,k$.
	
	Then, for two vertices $x$ and $y$, we define the geodesic time
	\begin{equation}
		t(x,y)= \inf_{\pi \in \cP(x,y)} T(\pi). \label{Définition geodesic time.}
	\end{equation}
	This defines a pseudometric on $\Z^d$ and this is a metric when the passage times only take positive values.
	A self-avoiding path $\gamma$ between $x$ and $y$ such that $T(\gamma)=t(x,y)$ is called a geodesic between $x$ and $y$.

	\paragraph*{Time constant in first-passage percolation.}
	
	The time constant describes the first-order of growth of the geodesic time. Assume that 
	\begin{equation}
		\E \min \left[T_1,\dots,T_{2d}\right] < \infty,\label{h: Hypothèse de moment.}
	\end{equation}
	where $T_1,\dots,T_{2d}$ are independent with distribution $\L$. Then, a subadditive argument gives that for every $x \in \R^d$, there exists a constant $\mu(x) \in [0,\infty)$ such that:
	\begin{equation}
		\lim\limits_{n\to\infty} \frac{t(0,nx)}{n} = \lim\limits_{n\to\infty} \frac{\E \left[t(0,nx)\right]}{n} = \mu (x) \text{ a.s.\, and in $L^1$.} \label{def: Définition de la constante de temps.}
	\end{equation}
	Note that in \eqref{def: Définition de la constante de temps.} and in the whole paper, for $x$ and $y$ in $\R^d$, we define $t(x,y)$ as $t(\lfloor x \rfloor,\lfloor y \rfloor)$ where $\lfloor x \rfloor$ is the unique vertex in $\Z^d$ such that $x \in \lfloor x \rfloor + [0,1)^d$ (similarly for $\lfloor y \rfloor$). 
	We refer to Theorem 2.18 in \cite{SaintFlourKesten} for this result which gives a first-order of growth of the geodesic time. For more details on the time constant and on first-passage percolation, we refer to \cite{50years}.
	
	\subsection{Euclidean length of geodesics in first-passage percolation}\label{Sous-section théorème 2.}
	
	The following theorem is the first result of this paper and is an extension of Theorem 2.5 in \cite{KRAS}. It gives a lower bound for the Euclidean length of geodesics in first-passage percolation. The assumptions on $\L$ are those of Theorem 1.5 in \cite{Jacquet2}. In particular, there is no moment assumption. Furthermore, the assumption on the weight of the minimum of the support of $\L$ (denoted by $\r$) is less restrictive than in Theorem 2.5 in \cite{KRAS}. Indeed we assume that $\L$ is useful, that is 
	\begin{equation}
		\begin{split}
			\L(\r) < p_c & \mbox{ when } \r=0, \\
			\L(\r) < \overrightarrow{p_c} & \mbox{ when } \r>0,
		\end{split}\label{eq: loi useful.}
	\end{equation}
	where  $p_c$ denotes the critical probability for Bernoulli bond percolation model on $\Z^d$ and $\overrightarrow{p_c}$ is the critical probability for oriented Bernoulli bond percolation on $\Z^d$ (we refer to \cite{Grimmett} for background on percolation and more precisely to Section 12.8 in \cite{Grimmett} for background on oriented Bernoulli bond percolation).
	
	This assumption guarantees that geodesics between any vertices exist almost surely (see for example Proposition 4.4 in \cite{50years}).
	
	
	For two vertices $x$ and $y$, when $\L(\infty)=0$, we define 
	\[\Lu(x,y) = \inf \{ |\gamma|_e \, : \, \text{$\gamma$ is a geodesic from $x$ to $y$}\},\]
	where for a finite path $\pi$, $|\pi|_e$ denotes the number of edges of $\pi$. In other words, $\Lu(x,y)$ is the minimal Euclidean length of a geodesic between $x$ and $y$.
	
	When $\L(\infty)>0$, there are vertices between which all paths have an infinite passage time. Assume that $\L([0,\infty)) > p_c$. Say that an edge $e$ is open if its passage time $T(e)$ is finite and closed otherwise. Thanks to this assumption, this percolation model is supercritical. We get that there exists a unique infinite component of open edges, which we denote by $\cC_\infty$. Then, we define the random set 
	\[\mathfrak{C} = \{(x,y) \in \Z^d \times \Z^d \, : \, \exists \text{ a path $\pi$ from $x$ to $y$ such that $T(\pi)<\infty$}\}.\]
	When $(x,y) \in \mathfrak{C}$, we can define 
	\[\Lu(x,y) = \inf \{ |\gamma|_e \, : \, \text{$\gamma$ is a geodesic from $x$ to $y$}\}.\]
	
	\begin{theorem}\label{thm: Théorème principal 2.}
		Assume that the support of $\L$ is included in $[0,\infty]$, that $\L$ is useful (i.e. that $\L$ satisfies \eqref{eq: loi useful.}) and that $\L([0,\infty)) > p_c$. 
		There exist deterministic constants $\alpha_1>0$, $\alpha_2>0$ and $\delta > 0$ such that for all $x \in \Z^d$, 
		\begin{equation}
			\P \left((0,x) \in \mathfrak{C} \text{ and } \Lu(0,x) \le (1+\delta) \|x\|_1 \right) \le \alpha_1 \mathrm{e}^{-\alpha_2 \|x\|_1}.\label{eq: equation du theoreme 2.}
		\end{equation}
	\end{theorem}
	
	The proof of Theorem \ref{thm: Théorème principal 2.} is given in Section \ref{Sect: Section preuve du théorème 2.}. We state the following immediate corollary. 
	
	\begin{corollary}
		Assume that the support of $\L$ is included in $[0,\infty]$, that $\L$ is useful (i.e. that $\L$ satisfies \eqref{eq: loi useful.}) and that $\L([0,\infty)) > p_c$.
		Then, there exists a deterministic constant $\delta>0$ and of an almost surely finite random constant $K$ such that $\Lu(0,x) \ge (1+\delta) \|x\|_1$ whenever $x \in \Z^d$ satisfies $\|x\|_1 \ge K$.
	\end{corollary}
	
	This corollary is proven, under slightly more restrictive assumptions in \cite{KRAS} (see Theorem 2.5).
	
	\subsection{Directed first-passage percolation}
	
	Denote by $\{\epsilon_1,\dots,\epsilon_d\}$ the vectors of the canonical basis. A directed path $\fpm =(x_0,\dots,x_k)$ is defined as a path such that for all $i=0,\dots,k-1$, there exists $j \in \{1,\dots,d\}$ such that $x_{i+1}=x_i+\epsilon_j$.
	For $x=(x^1,\dots,x^d)\in\R^d$ and $y=(y^1,\dots,y^d)\in\R^d$, we say that $x \le y$ (resp. $x \ge y$) if for all $j \in \{1,\dots,d\}$, $x^j \le y^j$ (resp. $x^j \ge y^j$).
	For two vertices $x$ and $y$ such that $x \le y$, we denote by $\fcP(x,y)$ the set of directed paths between $x$ and $y$. 
	
	Then, for two vertices $x$ and $y$ such that $x \le y$, we also define the directed geodesic time
	\begin{equation}
		\ft(x,y)= \inf_{\fpm \in \fcP(x,y)} T(\fpm). \label{Définition oriented geodesic time.}
	\end{equation}
	A directed path $\fg$ between $x$ and $y$ such that $T(\fg)=\ft(x,y)$ is called a directed geodesic between $x$ and $y$. Directed geodesics between any vertices $x$ and $y$ exist almost surely since there is a finite number of directed paths between $x$ and $y$.
	
	Assume now that $\E [ T(e) ] < \infty$ where $T(e)$ is a random variable with distribution $\L$. The same subadditive argument as the one used to get \eqref{def: Définition de la constante de temps.} allows us to describe the first-order of growth of the directed geodesic time. We get for every $x \in \R^d$ such that $x \ge 0$ the existence of a constant $\fm(x) \in [0,\infty)$ (called here the directed time constant) such that 
	\begin{equation}
		\lim\limits_{n\to\infty} \frac{\ft(0,nx)}{n} = \lim\limits_{n\to\infty} \frac{\E \left[\ft(0,nx)\right]}{n} = \fm (x) \text{ a.s.\, and in $L^1$.} \label{def: Définition de la constante de temps orientée.}
	\end{equation}
	We refer to \cite{MartinDirectedFPP} and to \cite{YuZhang} for the definition and for results on this directed time constant.
	
	\paragraph*{Comparison between the time constant and the directed time constant.}
	
	The second main result of this paper is the exponential bound for the comparison of the geodesic time and the directed geodesic time. It is given below in Theorem \ref{thm: Théorème principal bis 2.}, and the comparison between the two time constants in Theorem \ref{thm: Théorème principal.} follows from this result.
	The proof of Theorem \ref{thm: Théorème principal bis 2.} is based on the lower bound of the Euclidean length of geodesics given by Theorem \ref{thm: Théorème principal 2.}.
	
	Recall that we denote by  $\r$ the minimum of the support of $\L$ and that $\overrightarrow{p_c}$ is the critical probability for oriented Bernoulli bond percolation on $\Z^d$.
	We assume that
	\begin{equation}
		\L(\r) < \overrightarrow{p_c}.\label{h: hypothèse pour que la loi soit utile.}
	\end{equation}

	\begin{theorem}\label{thm: Théorème principal bis 2.}
		Assume that the support of $\L$ is included in $[0,\infty)$ and that \eqref{h: hypothèse pour que la loi soit utile.} holds. 
		Then, there exist constants $\delta>0$, $\alpha_1>0$ and $\alpha_2>0$ such that for all $x \in \R^d$ such that $x \ge 0$, 
		\begin{equation}
			\P \left( t(0,x) \le \ft(0,x) - \delta \|\lfloor x \rfloor \|_1 \right) \ge 1 - \alpha_1 \mathrm{e}^{-\alpha_2 \|x\|_1}.\label{eq: equation theoreme principal bis 2.}
		\end{equation}
	\end{theorem}
	
	\begin{theorem}\label{thm: Théorème principal.}
		Assume that the support of $\L$ is included in $[0,\infty)$, that \eqref{h: hypothèse pour que la loi soit utile.} holds and that $\E[T(e)] < \infty$ where $T(e)$ is a random variable with distribution $\L$. 
		Then, for every $x \in \R^d$ such that $x \ge 0$ and $x \ne 0$,
		\begin{equation}
			\mu(x) < \fm(x),\label{eq: équation théorème principal.}
		\end{equation}
		where $\mu(x)$ and $\fm(x)$ are defined at \eqref{def: Définition de la constante de temps.} and \eqref{def: Définition de la constante de temps orientée.}.
	\end{theorem}

	The proofs of Theorem \ref{thm: Théorème principal bis 2.} and \ref{thm: Théorème principal.} are given in Section \ref{Sect: Section proof.}.
	Note that since for all $x$ and $y$ in $\R^d$ such that $x \le y$, $\fcP(x,y) \subset \cP(x,y)$, we always have $t(x,y) \le \ft(x,y)$. It gives that for all $x \in \R^d$ such that $x \ge 0$, \[\mu(x) \le \fm(x),\] 
	and thus the result of Theorem \ref{thm: Théorème principal.} is on the strict inequality. 
	
	\begin{remk}\label{rem: lien avec l'article de KRAS.}
		When we make the slightly stronger assumption that the distribution $\L$ is useful (recall that the definition of a useful distribution is given at \eqref{eq: loi useful.}), Theorem \ref{thm: Théorème principal.}, with a weaker moment assumption\footnote{In \cite{KRAS}, the assumption is that $\displaystyle \E\left[(\min\{t_1,\dots,t_{2d}\})^d\right]<\infty$ where $t_1,\dots,t_{2d}$ are i.i.d.\ random variables with distribution $\L$.}, also follows from the results of Section 2.3 of the article of Krishnan, Rassoul-Agha and Seppäläinen \cite{KRAS}.  
		More precisely, using the notations of \cite{KRAS}, we get by Theorem 2.9 in \cite{KRAS} that for every $x \in \R^d$ such that $x \ge 0$ and $x \ne 0$, for each $\alpha>1$,
		\begin{equation}
			\mu(x) \le \alpha g^o\left(\frac{x}{\alpha}\right) \le \fm(x).\label{eq: equation remarque lien avec l'article de KRAS.}
		\end{equation}
		 From Theorem 2.11 in \cite{KRAS}, we obtain that the first inequality  in \eqref{eq: equation remarque lien avec l'article de KRAS.} is strict if $\alpha$ is close enough to $1$, which gives \eqref{eq: équation théorème principal.}.
	\end{remk}

	\begin{remk}
		Since Theorem \ref{thm: Théorème principal.} is an immediate corollary of Theorem \ref{thm: Théorème principal bis 2.}, one could improve the moment assumption in Theorem \ref{thm: Théorème principal.}. Here, we assume that $\E[T(e)] < \infty$ (where $T(e)$ is a random variable with distribution $\L$) for convenience for the definition of $\fm(x)$. 
	\end{remk}
	
	\paragraph*{Comments on the positivity of these constants.}
	
	In \cite{YuZhang}, Zhang proves results on the positivity of the directed time constant (defined here at \eqref{def: Définition de la constante de temps orientée.}) in dimension $2$.
	In particular, it is proven, when the passage times are nonnegative and when $\L$ has a finite first moment, that for all $x \in \R^d$ such that $x \ge 0$ and $x \ne 0$, $\fm(x)>0$ when $\L(0) < \overrightarrow{p_c}$. It is known in this setting that $\mu(x)=0$ for all $x \in \R^d$ if $\L(0) \ge p_c$, and that $\mu$ is a norm if $\L(0)<p_c$ (where $p_c$ denotes the critical probability for Bernoulli bond percolation model on $\Z^d$). Since Theorem \ref{thm: Théorème principal.} holds even if $\L(0) \in [p_c,\overrightarrow{p_c})$, it provides an alternative proof of the strict positivity of $\fm(x)$ for all $x \in \R^d$ such that $x \ge 0$ and $x \ne 0$ when $\L(0) < \overrightarrow{p_c}$ in any dimension $d$. 
	
	\subsection{Ideas of proofs}
	
	The proof of Theorem \ref{thm: Théorème principal 2.} is based on the notion of patterns developed in \cite{Jacquet2}. The idea is to consider a pattern in which the optimal path for the passage time between its endpoints is not optimal in terms of the number of edges. Then, using Theorem 1.5 in \cite{Jacquet2}, we get that every geodesic crosses a linear number of this pattern, which gives the desired lower bound in \eqref{eq: equation du theoreme 2.}.
	
	Theorem \ref{thm: Théorème principal.} is an easy consequence of Theorem \ref{thm: Théorème principal bis 2.}. The idea of the proof of Theorem \ref{thm: Théorème principal bis 2.} is to use the lower bound of the Euclidean length of geodesics in a shifted environment, that is an environment in which we add the same constant at the passage time of each edge of $\Z^d$. Since a directed geodesic between two vertices $x$ and $y$ such that $x \le y$ has exactly $\|y-x\|_1$ edges, we get a lower bound on the difference between the Euclidean length of a geodesic and of a directed geodesic.
	Then, we use the fact that, as in every environment, the geodesic time is lower than or equal to the directed geodesic time in the shifted environment. Finally, we get the difference between the geodesic time and the directed geodesic time desired in \eqref{eq: equation theoreme principal bis 2.} in the initial environment using the three following tools:
	\begin{enumerate}[label=(\roman*)]
		\item we have a lower bound on the difference of the number of edges between a geodesic and a directed geodesic in the shifted environment,
		\item a path is a directed geodesic in the shifted environment if and only if it is a directed geodesic in the initial one,
		\item the difference of the passage time of a path between the initial environment and the shifted one is proportional to its number of edges.
	\end{enumerate} 
	
	\paragraph*{Comments on the assumptions of Theorem \ref{thm: Théorème principal bis 2.} and Theorem \ref{thm: Théorème principal.}.}
	
	Note that in Theorem \ref{thm: Théorème principal 2.}, it is assumed that the support of $\L$ is included in $[0,\infty]$ and that $\L$ is useful (i.e.\ $\L$ satisfies \eqref{eq: loi useful.}).
	But since this theorem is used in a shifted environment in the proof of Theorem \ref{thm: Théorème principal bis 2.}, it allows us to consider in Theorem \ref{thm: Théorème principal bis 2.} and Theorem \ref{thm: Théorème principal.} distributions such that $\L(\r) < \overrightarrow{p_c}$ instead of $\L(\r) < p_c$.
	
	\section{Proofs of Theorem \ref{thm: Théorème principal bis 2.} and Theorem \ref{thm: Théorème principal.}}\label{Sect: Section proof.}
	
	\begin{proof}[Proof of Theorem \ref{thm: Théorème principal bis 2.} using Theorem \ref{thm: Théorème principal 2.}]
	
	Assume that $\L$ satisfies the assumptions of Theorem \ref{thm: Théorème principal bis 2.}.
	Fix $\Delta > 0$. 
	
	For each environment  $T=\{T(e) \, : \, e \in \cE\}$ of independent random variables with distribution $\L$, we define the environment $T_\Delta$ as the environment in which for every $e \in \Z^d$, 
	\begin{equation}
		T_\Delta(e) = T(e) + \Delta.\label{eq: Définition de l'environnement shifté.}
	\end{equation}
	Recall that for every $x \in \R^d$, we denote by $\lfloor x \rfloor$ the unique vertex in $\Z^d$ such that $x \in \lfloor x \rfloor + [0,1)^d$.
	For every $x \in \R^d$ such that $x \ge 0$, we denote by $\gamma_\Delta(x)$ the first geodesic in any fixed deterministic order from $0$ to $\lfloor x \rfloor$ in the shifted environment $T_\Delta$, and by $|\gamma_\Delta(x)|_e$ the number of edges of $\gamma_\Delta(x)$. The environment $T_\Delta$ satisfies the assumptions guaranteeing the existence of geodesics between any vertices almost surely.  
	Then, we denote by $\fg_\Delta(x)$ the first directed geodesic in any deterministic order from $0$ to $\lfloor x \rfloor$ in the shifted environment $T_\Delta$. 
	
	Denote by $\L^\Delta$ the distribution of $T_\Delta(e)$. 
	Using the assumptions on $\L$, we have that the support of $\L^\Delta$ is included in $[0,\infty)$ and that $\L^\Delta(\r^\Delta) < \overrightarrow{p_c}$ where $\r^\Delta>0$. Thus, $\L^\Delta$ is useful and satisfies the assumptions of Theorem \ref{thm: Théorème principal 2.}. By Theorem \ref{thm: Théorème principal 2.}, we get constants $\alpha_1>0$, $\alpha_2>0$ and $\delta > 0$ such that for every $x \in \R^d$, 
	\begin{equation}
		\P \left( |\gamma_\Delta(x)|_e \ge (1+\delta)\|\lfloor x \rfloor\|_1 \right) \ge 1 - \alpha_1 \mathrm{e}^{-\alpha_2 \| x \|_1}.\label{eq: équation minoration proba d'avoir suffisamment d'arêtes.}
	\end{equation}

	\vspace{\baselineskip}
	
	Now, let $x \in \R^d$ such that $x \ge 0$. 
	Assume that $|\gamma_\Delta(x)|_e \ge (1+\delta)\|\lfloor x \rfloor\|_1$.
	Then we have the following sequence of inequalities, whose justifications are given just below:

	\begin{align}
		t(0,x) & \le T(\gamma_\Delta(x)) \label{eq: numéro 1 dans le align.} \\
		& = T_\Delta(\gamma_\Delta(x)) - \Delta |\gamma_\Delta(x)|_e \label{eq: numéro 2 dans le align.} \\
		& \le T_\Delta(\gamma_\Delta(x)) - \Delta (1+\delta) \|\lfloor x \rfloor\|_1 \nonumber \text{ using the assumption $|\gamma_\Delta(x)|_e \ge (1+\delta)\|\lfloor x \rfloor\|_1$}, \\
		& \le T_\Delta(\fg_\Delta(x)) - \Delta (1+\delta) \|\lfloor x \rfloor\|_1 \label{eq: numéro 3 dans le align.} \\
		& \le \underbrace{T_\Delta(\fg_\Delta(x)) - \Delta |\fg_\Delta(x)|_e}_{=T(\fg_\Delta(x))} - \Delta \delta \|\lfloor x \rfloor\|_1 \label{eq: numéro 4 dans le align.} \\
		& = \ft(0,nx) - \Delta \delta \|\lfloor x \rfloor\|_1. \label{eq: numéro 5 dans le align.}
	\end{align}	

	The inequality \eqref{eq: numéro 1 dans le align.} comes from the fact that $\gamma_\Delta(x) \in \cP(0,\lfloor x \rfloor)$ and from the definition of the geodesic time given at \eqref{Définition geodesic time.}. For \eqref{eq: numéro 2 dans le align.}, we use the definition of the shifted environment given at \eqref{eq: Définition de l'environnement shifté.}. Indeed, with this definition, for every path $\pi$, we have 
	\begin{equation}
		T_\Delta (\pi) = T(\pi) + \Delta |\pi|_e.\label{eq: égalité temps dans l'environnement shifté et temps dans l'environnement normal.}
	\end{equation}
	We get \eqref{eq: numéro 3 dans le align.} using that, since $\fcP(0,\lfloor x \rfloor) \subset \cP (0,\lfloor x \rfloor)$, a geodesic from $0$ to $\lfloor x \rfloor$ has a lower passage time than a directed geodesic from $0$ to $\lfloor x \rfloor$. For \eqref{eq: numéro 4 dans le align.}, we use the fact that every path $\fcP(0,\lfloor x \rfloor)$ has exactly $\|\lfloor x \rfloor\|_1$ edges and thus that $|\fg_\Delta(nx)|_e = \|\lfloor x \rfloor\|_1$. Then, we use again \eqref{eq: égalité temps dans l'environnement shifté et temps dans l'environnement normal.}. Finally, using again these arguments, a path is a directed geodesic from $0$ to $\lfloor x \rfloor$ in the environment $T$ if and only if it is a directed geodesic from $0$ to $\lfloor x \rfloor$ in the environment $T_\Delta$. Hence $T(\fg_\Delta(x)) = \ft(0,x)$. 
	
	\vspace{\baselineskip}
	
	From the sequence of inequalities above, we get
	\begin{equation}
		\left\{ |\gamma_\Delta(x)|_e \ge (1+\delta)\|\lfloor x \rfloor\|_1 \right\} \subset \left\{ t(0,x) \le \ft(0,nx) - \Delta \delta \|\lfloor x \rfloor\|_1 \right\},
	\end{equation}
	which gives \eqref{eq: equation theoreme principal bis 2.} using \eqref{eq: équation minoration proba d'avoir suffisamment d'arêtes.}.
	\end{proof}

	\begin{proof}[Proof of Theorem \ref{thm: Théorème principal.}]
	
	Assume that $\L$ satisfies the assumptions of Theorem \ref{thm: Théorème principal.}. In particular, $\L$ satisfies the assumptions of Theorem \ref{thm: Théorème principal bis 2.}.
	Fix $x\in\R^d$ such that $x \ge 0$ and $x \ne 0$. Using that $\E[T(e)] < \infty$, where $T(e)$ is a random variable with distribution $\L$, and using Theorem \ref{thm: Théorème principal bis 2.}, we get the existence of a constant $\delta'>0$ such that for every $n$ sufficiently large, 
	\begin{equation}
		\E \left[ t(0,nx) \right] \le \E \left[ \ft(0,nx) \right] - \delta' \| \lfloor nx \rfloor \|_1.\label{eq: équation avec les espérances.}
	\end{equation}
	Dividing by $n$ in the inequality \eqref{eq: équation avec les espérances.}, we get
	\[\frac{\E \left[ t(0,nx) \right]}{n} \le \frac{\E \left[ \ft(0,nx) \right]}{n} - \delta' \left(\|x\|_1 - \frac{d}{n}\right).\]
	Finally, taking the limit in each side, we obtain
	\[\mu(x) = \lim\limits_{n \to \infty} \frac{\E \left[ t(0,nx) \right]}{n} < \lim\limits_{n \to \infty} \frac{\E \left[ \ft(0,nx) \right]}{n} = \fm(x),\]
	which concludes the proof of Theorem \ref{thm: Théorème principal.}.
	\end{proof}

	\section{Proof of Theorem \ref{thm: Théorème principal 2.}}\label{Sect: Section preuve du théorème 2.}
	
	\begin{proof}[Proof of Theorem \ref{thm: Théorème principal 2.}]
		This proof is based on the notion of patterns defined in \cite{Jacq} and follows the ideas of the proof of Section 4.2 in \cite{Jacq}. 
		Assume that $\L$ satisfies the assumptions of Theorem \ref{thm: Théorème principal 2.}.
		
		\begin{itemize}
			\item Assume first that $\L$ has at least two finite points in its support ant let $0 \le a < b < \infty$ two points in the support of $\L$. The case where there exists $a' \in \R$ such that $\L(a') + \L(\infty) = 1$ is dealt with at the end of the proof. When $\L(\infty)=0$, the existence of at least two different finite points in the support of $\L$ is guaranteed by the fact that $\L$ is useful.
			Fix 
			\begin{equation}
				\l > \frac{2a}{b-a}.\label{eq: on fixe l.}
			\end{equation}
			We define the subset $\displaystyle \Lambda = \{0,\dots,\l\} \times \{0,1\} \times \prod_{i=3}^d \{0\}$, and the vertices $\ulm = (0,\dots,0)$ and $\vlm = (\l,0,\dots,0)$.
			We denote by $\cE_\Lambda$ the set of edges linking vertices of $\Lambda$.
			We define the path $\PP$ as the path going from $\ulm$ to $\vlm$ by $\l$ steps in the direction $\epsilon_1$ and the path $\PPP$ as the path going from $\ulm$ to $\ulm + \epsilon_2$ by one step in the direction $\epsilon_2$, then to $\vlm + \epsilon_2$ by $\l$ steps in the direction $\epsilon_1$ and then to $\vlm$ by one step in the direction $\epsilon_2$. 
			
			For a deterministic family $(t_e)_{e \in \cE_\Lambda}$ of passage times on the edges of $\Lambda$ and for a path $\pi$, we denote $\displaystyle \sum_{e\in\pi} t_e$ by $T(\pi)$. For all $\delta \ge 0$, we consider the set $G(\delta)$ of families $(t_e)_{e \in \cE_\Lambda}$ of passage times on the edges of $\Lambda$ which satisfy the following two conditions:
			\begin{itemize}
				\item for all $e \in \PPP$, $t_e \in [a-\delta,a+\delta]$,
				\item for all $e \in \Lambda \setminus \PPP$, $t_e \in [b-\delta,b+\delta]$.
			\end{itemize}
			Then, consider the set $H$ of families $(t_e)_{e \in \cE_\Lambda}$ such that $\PPP$ is the unique optimal path from $\ulm$ to $\vlm$ among the paths entirely contained in $\Lambda$. 
			
			\begin{claim}\label{c: claim 1.}
				There exists $\delta>0$ such that $G(\delta) \subset H$.
			\end{claim}
			
			The proof of the claim is given below and, for now, we assume the claim to be true.
			
			Using Claim \ref{c: claim 1.}, fix $\delta>0$ such that $G(\delta) \subset H$.
			Define the event $\cA^\Lambda=\{(T(e))_{e\in\cE_\Lambda} \in G(\delta)\}$.
			Then $\mathfrak{P}=(\Lambda,\ulm,\vlm,\cA^\Lambda)$ is a valid pattern in the terminology of \cite{Jacquet2} (see Definition 1.2 in \cite{Jacquet2}). Indeed, $a \ne \infty$ and $b \ne \infty$, and since $a$ and $b$ belong to the support of $\L$, we have that $\P((T(e))_{e\in\cE_\Lambda} \in G(\delta))>0$.
			
			\vspace{\baselineskip}
			
			For every $x \in \Z^d$, if $(0,x) \in \mathfrak{C}$, we denote by $\gamma_-(x)$ the first geodesic in the lexicographical order between $0$ and $x$ among those whose number of edges is equal to $\Lu(0,x)$.
			We denote by $N^\mathfrak{P}(\gamma_-(x))$ the number of patterns $\mathfrak{P}$ visited by $\gamma_-(x)$. 
			Since the distribution $\L$ is useful and $\L([0,\infty))>p_c$, we can apply Theorem 1.5 in \cite{Jacquet2} to get constants $\delta,\beta_1,\beta_2>0$ such that
			\begin{equation}
				\P \left((0,x) \in \mathfrak{C} \text{ and } N^\mathfrak{P}(\gamma_-(x)) \le \delta \|x\|_1 \right) \le \beta_1 \mathrm{e}^{-\beta_2 \|x\|_1}.\label{eq: utilisation du théorème des motifs.}
			\end{equation} 
		
			\begin{claim}\label{c: claim 2.}
				For all vertices $z_1$ and $z_2$ and every path $\pi$ between $z_1$ and $z_2$, we have 
				\[|\pi|_e \ge \| z_2 - z_1 \|_1 + N^\mathfrak{P}(\pi).\]
			\end{claim}
		
			The proof of this claim, given below, is based on easy geometrical considerations. For now, assume the claim to be true.
		
			When $(0,x) \in \mathfrak{C}$, by Claim \ref{c: claim 2.}, $|\gamma_-(x)|_e \ge \|x\|_1 +  N^\mathfrak{P}(\gamma_-(x))$. Thus, we get
			\begin{align*}
				\P((0,x) \in \mathfrak{C} \text{ and } \Lu(0,x) \le (1+\delta) \|x \|_1) &= \P ((0,x) \in \mathfrak{C} \text{ and } |\gamma_-(x)|_e \le (1+\delta)\|x\|_1) \\
				& \le \P ((0,x) \in \mathfrak{C} \text{ and } N^\mathfrak{P}(\gamma_-(x)) \le \delta \|x\|_1) \\
				& \le \beta_1 \mathrm{e}^{-\beta_2 \|x\|_1} \text{ by \eqref{eq: utilisation du théorème des motifs.},}
			\end{align*}
			which concludes the proof.
			
			
			\vspace{\baselineskip}
			
			It remains to prove Claim \ref{c: claim 1.} and Claim \ref{c: claim 2.}.
			
			\begin{proof}[Proof of Claim \ref{c: claim 1.}]
				Let us first prove that $G(0) \subset H$. Consider a family $(t_e)_{e \in \cE_\Lambda} \in G(0)$. We have $T(\PPP) = (\l+2)a$ and $T(\PP) = \l b$. By \eqref{eq: on fixe l.}, \[T(\PPP) < T(\PP).\]
				The path $\PP$ is the only path entirely contained in $\Lambda$ which does not take edges in the direction $\epsilon_2$. Let $\pi$ be a path from $\ulm$ to $\vlm$ entirely contained in $\Lambda$ different from $\PP$ and $\PPP$. For a path $\pi'$, we denote by $T^1(\pi')$ (resp. $T^2(\pi')$) the sum of the passage times of the edges of $\pi'$ which are in the direction $\epsilon_1$ (resp. $\epsilon_2$). Since the paths entirely contained in $\Lambda$ only take edges in the directions $\epsilon_1$ and $\epsilon_2$, we have \[T(\PPP)=T^1(\PPP)+T^2(\PPP) \text{ and } T(\pi)=T^1(\pi)+T^2(\pi).\]
				Now, since $\pi$ is different from $\PPP$, $\pi$ takes some edge in the directio, $\epsilon_2$ different from $\{\ulm,\ulm+\epsilon_2\}$ and $\{\vlm,\vlm+\epsilon_2\}$. Furthermore, $\pi$ has to take an even number of edges in the direction $\epsilon_2$. Since the only edges in the direction $\epsilon_2$ whose passage time is equal to $a$ are $\{\ulm,\ulm+\epsilon_2\}$ and $\{\vlm,\vlm+\epsilon_2\}$, it gives 
				\begin{equation}
					T^2(\pi) > T^2(\PPP).\label{eq: équation 1 dans la preuve du lemme.}
				\end{equation}
				Then, $\pi$ also has to take at least $\l$ edges in the direction $\epsilon_1$. Thus,
				\begin{equation}
					T^1(\pi) \ge \l a = T^1(\PPP).\label{eq: équation 2 dans la preuve du lemme.}
				\end{equation}
				Combining \eqref{eq: équation 1 dans la preuve du lemme.} and \eqref{eq: équation 2 dans la preuve du lemme.} yields $T(\PPP) < T(\pi)$, which proves that $(t_e)_{e \in \cE_\Lambda} \in H$ and that $G(0) \subset H$. 
				
				\vspace{\baselineskip}
				
				Then, $H$ is an open set since for a family $(t_e)_{e \in \cE_\Lambda}$ to belong to $H$, it is required that the time of one path is strictly smaller than the time of every path of a finite family of paths. Hence, for $\delta > 0$ small enough, we have 
				\begin{equation}
					G(\delta) \subset H.\label{eq: équation 3 dans la preuve du lemme.}
				\end{equation}
			\end{proof}
			
			\begin{proof}[Proof of Claim \ref{c: claim 2.}]
				Consider two vertices $z_1$ and $z_2$ in $\Z^d$ and a path $\pi$ between $z_1$ and $z_2$. For every $j \in \{1,\dots,d\}$, denote by $z_1(j)$ and $z_2(j)$ the $j$-th coordinates of $z_1$ and $z_2$, and let $\chi_j=|z_1(j)-z_2(j)|$. The path $\pi$ has to take at least $\chi_j$ edges in the direction $\epsilon_j$ for every $j \in \{1,\dots,d\}$. This gives \[|\pi|_e \ge \sum_{j=1}^d \chi_j = \| z_2 - z_1 \|_1.\]
				More precisely, the $\chi_2$ edges in the direction $\epsilon_2$ the path $\pi$ has to take are $\chi_2$ edges linking the hyperplanes $H^2_\l=\{(i_1,\dots,i_d)\in\Z^d \, : \, i_2=\l\}$ and $H^2_{\l+1}=\{(i_1,\dots,i_d)\in\Z^d \, : \, i_2=\l+1\}$ for every $\l \in \{0,\dots,\chi_2-1\}$. But, each time $\pi$ crosses a pattern $\mathfrak{P}$, $\pi$ takes two edges linking the same hyperplanes $H^2_\l=\{(z_1,\dots,z_d)\in\Z^d \, : \, z_2=\l\}$ and $H^2_{\l+1}=\{(z_1,\dots,z_d)\in\Z^d \, : \, z_2=\l+1\}$ for some $\l \in \Z$. This yields 
				\begin{equation}
					|\pi|_e \ge \| z_2 - z_1 \|_1 + N^\mathfrak{P}(\pi).\label{eq: équation sur le nombre d'arêtes de la géodésique.}
				\end{equation}
			\end{proof}
	
		\item It remains to deal with the case when there exists $a'\in\R$ such that $\L(a') + \L(\infty)=1$ but this case is simpler than the others. We replace the pattern $\mathfrak{P}$ in the proof above by the pattern $\mathfrak{P}_\infty = (\Lambda_\infty,\ulm_\infty,\vlm_\infty,\cA^\Lambda_\infty)$ where $\displaystyle \Lambda_\infty=\{0,1\}^2 \times \prod_{i=3}^d \{0\}$, $\ulm_\infty=(0,\dots,0)$, $\vlm_\infty=(1,0,\dots,0)$ and $\cA^\Lambda_\infty$ is the event on which $T(e)=a'$ for every edge $e \in \Lambda_\infty$ different from $\{\ulm_\infty,\vlm_\infty\}$, and $T(\{\ulm_\infty,\vlm_\infty\})=\infty$. Then the proof follows the one of the case $\L(\infty)>0$ above but with this pattern.
		\end{itemize} 
	\end{proof}
	
	
	
	\paragraph*{Acknowledgements.}
		This work has been supported by the project ANR MISTIC (ANR-19-CE40-0005).
		I would like to thank Arjun Krishnan, Firas Rassoul-Agha and Timo Seppäläinen for pointing out that Theorem \ref{thm: Théorème principal.} follows also from their work, as it is explained in Remark \ref{rem: lien avec l'article de KRAS.}.
		I would also like to thank Olivier Durieu and Jean-Baptiste Gouéré for their many suggestions, which helped to improve this short paper.

	\bibliographystyle{plain}
	
\end{document}